\documentclass{amsart}
\usepackage{amssymb,amsmath,latexsym,mathrsfs}
\usepackage{times}

\newcommand{\showhide}[1]{#1} 
\showhide{\usepackage{tikz}}

\numberwithin{equation}{section}

\newtheorem{thm}{Theorem}[section]

\newtheorem{cor}[thm]{Corollary}

\newtheorem{prop}[thm]{Proposition}
\newtheorem{lemma}[thm]{Lemma}

\newcommand{\cl}{\mathrm{cl}}

\newcommand{\rk}{\mathrm{rk}}

\title{Syzygies on Tutte polynomials of freedom matroids}
\author{Joseph P.S. Kung}

\begin{document}

\begin{abstract}    
It follows from a theorem of H. Derksen [{\it J. Algebraic Combin.,}  30 (2009) 43--86] that the Tutte polynomial of a rank-$r$ matroid 
on an $n$-set is ``naturally'' a linear combination of Tutte polynomials of rank-$r$ size-$n$ freedom matroids.  However, the Tutte polynomials of rank-$r$ size-$n$ freedom matroids are not linearly independent.  We construct two natural bases for these polynomials and as a corollary, we prove that the Tutte polynomials of rank-$r$ matroids of size-$n$ spans a subspace of dimension $r(n-r)+1.$  We also find a generating set for the linear relations between Tutte polynomials of freedom matroids.  This generating set is indexed by a pair of intervals, one of size $2$ and one of size $4,$ in the weak order 
of freedom matroids.  This weak order is a distributive lattice and a sublattice of Young's partition lattice.    
\end{abstract}

\maketitle


\maketitle

\section{Two matroid invariants}

We begin with the cadet, the $\mathcal{G}$-invariant, introduced by Derksen \cite{Derksen} in 2009.     
Let $M$ be an $(n,r)$-matroid, that is, a rank-$r$ matroid 
on the set $\{1,2,\ldots, n\}$ with rank function $\rk$ and closure $\cl.$  For a permutation $\pi$ on $\{1,2, \ldots, n\},$  the {\em rank sequence} $\underline{r}(\pi)$ of $\pi$ is the sequence   $ r_1 r_2 \ldots r_n$ 
defined by $r_1 = \rk(\{\pi(1)\})$ and for $j \geq 2,$
\[
r_j = \rk(\{\pi(1), \pi(2),\ldots,\pi(j)\}) - \rk(\{\pi(1), \pi(2),\ldots,\pi(j-1)\}). 
\]
It is immediate that $r_j = 0$ or $1,$ there are exactly $r$ $1$'s, and the set 
%
 $\{\pi(j): r_j =1\}$  
is a basis  of $M.$   

A {\sl bit sequence} is a sequence of zeros and ones, and an {\sl $(n,r)$-sequence} is a bit sequence of length $n$ with (exactly) $r$ $1$'s.  Let $[\underline{r}]$ be a variable or formal symbol, one for each $(n,r)$-sequence $\underline{r},$ and $\mathcal{G}(n,r)$ be the vector space of dimension $\binom {n}{r}$ consisting of all formal linear combination  
of symbols $[\underline{r}]$ with coefficients in a field $\mathbb{K}$ of characteristic $0.$   
The {\em $\mathcal{G}$-invariant} $\mathcal{G}(M)$ and its coefficients $g_{\underline{r}}(M)$ are defined by 
\[
\mathcal{G}(M) = \sum_{\pi}  [\underline{r}(\pi)] = \sum_{\underline{r}} g_{\underline{r}}(M) [\underline{r}],
\]
where the first sum ranges over all $n!$ permutations of $\{1,2,\ldots,n\}.$ 
A {\em specialization} of the $\mathcal{G}$-invariant taking values in an abelian group $\mathbb{A}$ is a function assigning a value in $\mathbb{A}$ to each symbol $[\underline{r}].$  
The $\mathcal{G}$-invariant is fundamental because by a theorem of Derksen and Fink \cite{DerksenFink},   
it is a universal valuative invariant on matroid base polytopes, in the sense that every 
valuative invariant on base polytopes is a specialization of the $\mathcal{G}$-invariant.    

The veteran is the Tutte polynomial.  It is a classical and well-studied object.      
To clarify notation, we recall the definition of the Tutte polynomial: for a rank-$r$ matroid $M$ on a set $S$ with $n$ elements, the Tutte polynomial $T(M)$ and its coefficients $t_{ij}(M)$ are defined by   
\begin{eqnarray*}
T(M) = T(M;x,y) 
& = &  
\sum_{i,j \geq 0}  t_{ij} (M) x^i y^j 
\\
& = &  \sum_{A \subseteq S}  (x-1)^{r - \rk (A)} (y-1)^{|A| - \rk(A)}. 
\end{eqnarray*}
We denote by $\mathcal{T}(n,r)$ the subspace in the algebra $\mathbb{K}[x,y]$ of polynomials in the variables $x$ and $y$ with coefficients in the field $\mathbb{K}$ spanned by the Tutte polynomials of $(n,r)$-matroids.    
We assume a basic acquaintance with the theory of Tutte polynomials (see \cite{BryOx, Handbook} for surveys).    

Derksen \cite{Derksen} showed that there is a specialization sending the $\mathcal{G}$-invariant to the Tutte polynomial.  
The specialization is given explicitly in the following lemma. 

\begin{lemma} \label{specialization}    The assignment $\mathcal{G}(n,r) \to \mathbb{K}[x,y],$
\[
[r_1r_2 \ldots r_n] \mapsto  \sum_{m=0}^n  
\frac {(x-1)^{r - \mathrm{wt}(r_1r_2 \ldots r_m)}(y-1)^{m-\mathrm{wt}(r_1r_2 \ldots r_m)}}{m! (n-m)!} ,
\] 
where the Hamming weight $\mathrm{wt}(r_1r_2 \ldots r_m)$ is the number of $1$'s in the initial segment $r_1r_2 \ldots r_m,$ sends the $\mathcal{G}$-invariant of a matroid to its Tutte polynomial.   
Explicitly, 
\[
T(M) = \frac {1}{n!} \sum_{\pi} \left(  \sum_{m=0}^n \binom {n}{k} (x-1)^{r - \rk(\{\pi(1),\pi(2),\ldots,\pi(m)\})}(y-1)^{m-\rk(\{\pi(1),\pi(2),\ldots,\pi(m)\})} \right).
\]
\end{lemma}

This specialization extends to a linear transformation $\mathsf{Sp}: \mathcal{G}(n,r) \to \mathbb{K}[x,y].$   
Our objective in this paper is to determine the kernel of $\mathsf{Sp},$ show that its image is $\mathcal{T}(n,r),$ and describe natural bases for $\mathcal{T}(n,r).$       
To do this, we need a partial order on $(n,r)$-sequences.   This order is described in Section \ref{Younglattice}.  With this order, we describe  in Section \ref{szkernel} the syzygies of $\mathsf{Sp},$ that is, linear combinations in $\mathrm{ker}\,\mathsf{Sp}.$  Freedom matroids are introduced in Section \ref{freedom}.  The $\mathcal{G}$-invariants of freedom matroids form a natural basis of $\mathcal{G}(n,r).$  Hence the Tutte polynomials of freedom matroids span $\mathcal{T}(n,r);$  however, they fail to form a basis.  We describe two subsets of freedom matroids whose Tutte polynomials form a basis in Section \ref{Twobases} and a generating set for linear relations on Tutte polynomials of freedom matroids in Section \ref{LR}.  In Section \ref{coefficients},  we use a basis found in Section \ref{Twobases} to give another proof of a theorem of Brylawski \cite{Brylawski} describng a basis for linear relations on coefficients of Tutte polynomials.  
The last two sections are computational.   We give formulas for Tutte polynomials in one of the bases found in Section \ref{Twobases} and as an example, compute explicitly the freedom matroids and their relations in $\mathcal{T}(5,3).$  

\section{A partial order on rank  sequences} \label{Younglattice} 

Let $\mathcal{S}(n,r)$ be the set of $(n,r)$-sequences.  We define the (partial) order $\trianglerighteq$ in the following way. If $\underline{r}$ and $\underline{s}$ are two $(n,r)$-sequences, then 
$\underline{s} \trianglerighteq \underline{r}$ if for every index $j,\, 1 \leq j \leq n,$ 
\[
s_1 + s_2 + \cdots + s_j   \geq   r_1 + r_2 + \cdots + r_j,
\] 
in other words, reading from the left, there are always at least as many $1$'s in $\underline{s}$ as there are in $\underline{r}.$    
Using the notation where $1^a$ stands for a sequence of $a$ (consecutive) $1$'s and $0^b$ a sequence of $b$ $0$'s, this order has maximum $1^r0^{n-r}$ and minimum $0^{n-r}1^r.$     
The partial order $(\mathcal{S}(n,r),\trianglerighteq)$ is a sublattice of Young's (partition) lattice (see, for example, \cite[p.~288]{EC2}).      We shall use $\trianglerighteq$ as the underlying order for ``straightening'' or Gr\"obner basis arguments.  In particular, no esoteric properties of Young's lattice will be used.  

An intuitive way to think of the order $\trianglerighteq$ is to view a sequence $\underline{r}$ as a lattice path from the origin $(0,0)$ to the corner $(r,n-r)$ where a $1$ is a north step and a $0$ is an east step.  Then 
$\underline{s} \trianglerighteq \underline{r}$ 
if and only if as lattice paths, $\underline{r}$ never goes higher than $\underline{s}.$   The lattice paths lie inside the rectangle with opposite corners $(0,0)$ and $(r,n-r).$ 
Tilting the rectangle so that it pirouettes on the corner $(0,0),$ we can also think of a sequence 
$\underline{r}$ as an order ideal on the direct product of a $r$-chain and an $(n-r)$-chain, with $\trianglerighteq$ equal to set-containment $\subseteq \!.$  As order ideals are subsets closed under intersections and unions, $(\mathcal{S}(n,r),\trianglerighteq)$  is a distributive lattice with meet equal to intersection and join equal to union.   

The sequence $\underline{r}$ has an {\em ascent} at position $i$ if $i \geq 2,$ $r_{i-1} = 0,$ and $r_i = 1.$  It has a {\em descent} at $i$ if $i \geq 2,$ $r_{i-1} = 1,$ and $r_i = 0.$  
In $(\mathcal{S}(n,r),\trianglerighteq),$ $\underline{s}$ covers $\underline{r}$ if 
\[
\underline{s} = \underline{r}_1 10 \underline{r}_2, \,\, \underline{r} = \underline{r}_1 01 \underline{r}_2 . 
\] 
An element $j$ in a lattice is a {\em join-irreducible} if $j$ covers at most one element.  In the lattice $(\mathcal{S}(n,r),\trianglerighteq),$ a sequence is join-irreducible if and only if it has at most one descent.  In particular, join-irreducibles are sequences of the form $0^a 1^b 0^c 1^d,$ with $a+b+c+d = n$ and $b+d = r.$    
Working upside-down, an element $m$ in a lattice is a {\em meet-irreducible} if $m$ covers at most one element.  A sequence is meet-irreducible if and only if it has at most one ascent and meet-irreducibles are sequences of the form $1^a 0^b 1^c 0^d,$ with $a+b+c+d =n$ and $a +c = r.$  We denote the set of join-irreducibles (respectively, meet-irreducibles) in $(\mathcal{S}(n,r),\trianglerighteq)$ by $\mathsf{J}$ (respectively, $\mathsf{M}$).  
A simple counting argument gives   
\[
|\mathsf{J}| = r(n-r)+1 = |\mathsf{M}|. 
\] 

By definition, $g_{\underline{r}}(M) \geq 0.$  The {\sl support} $\mathrm{supp}(M)$ of an $(n,r)$-matroid $M$ is the set $\{\underline{r}: g_{\underline{r}}(M) > 0 \}$.  

\begin{lemma}\label{support}
The support of an $(n,r)$-matroid $M$ is an order filter in $(\mathcal{S}(n,r),\trianglerighteq).$
\end{lemma}
\begin{proof} 
It suffices to show that if $\underline{s}$ covers $\underline{r},$ and $\underline{r} \in \mathrm{supp}(M),$ then 
$\underline{s} \in \mathrm{supp}(M).$  To do this, let $\underline{r}=\underline{r}_1 01 \underline{r}_2,$  and $\underline{s}=\underline{r}_1 10 \underline{r}_2,$ where $\underline{r}_1$ has length $\lambda.$   
As $g_{\underline{r}}(M) > 0,$  it is the rank sequence of a permutation $i_1 i_2 \ldots i_n$ (in one-line notation).  Then $\underline{s}$ is the rank sequence of the permutation $i_1 i_2 \ldots i_{\lambda} i_{\lambda+2}i_{\lambda+1} \ldots i_n$ and we conclude that $g_{\underline{s}}(M) \geq 1.$  
\end{proof}

\section{Syzygies for $\ker \mathsf{Sp}$}  \label{szkernel} 

The cornerstone of our theory is the following  lemma.  

\begin{lemma} \label{cornerstone}  
Let $\underline{r}$ be a $(\lambda,\rho)$-sequence and $\underline{s}$ a bit sequence such that $\underline{r} 01 \underline{s}$ is an $(n,r)$-sequence.  Then   
\[ 
\mathsf{Sp} ( [\underline{r} 10 \underline{s}] - [\underline{r} 01 \underline{s}])   
=  
\frac {(x-1)^{r- \rho -1} (x + y - xy)  (y-1)^{\lambda - \rho - 1}}{(\lambda + 1)!(n-\lambda -1)!}.
\] 
\end{lemma}

\begin{proof} We use Lemma \ref{specialization}, noting that the summands in 
$\mathsf{Sp} ([\underline{r} 10 \underline{s}])$ and $\mathsf{Sp} ( [\underline{r} 01 \underline{s}])$ are the same except at 
$m = \lambda + 1.$  When $m = \lambda + 1,$  the summands for $\mathsf{Sp} ( [\underline{r} 10 \underline{s}])$ and  $\mathsf{Sp} ( [\underline{r} 01 \underline{s}])$ are 
\[
\frac {(x-1)^{r- \rho -1} (y-1)^{\lambda - \rho}}{(\lambda + 1)!(n-\lambda -1)!} \quad \mathrm{and} \quad
\frac {(x-1)^{r- \rho}   (y-1)^{\lambda - \rho+1}}{(\lambda + 1)!(n-\lambda -1)!}
\]   
and the lemma follows.  
\end{proof} 

An interval $I$ of height $2$  in $(\mathcal{S}(n,r),\trianglerighteq)$ 
has the form shown in Figure \ref{fig:interval}, 
\showhide{
\begin{figure}
  \centering
  \begin{tikzpicture}[scale=1]
%
\node[inner sep = 0.3mm] (b) at (7.75,-0.5) {\small $\underline{r}_1 01 \underline{r}_2 01 \underline{r}_3$};
  \node[inner sep = 0.3mm] (l) at (6.5,0.3) {\small $\underline{r}_1 01 \underline{r}_2 10\underline{r}_3$};
  \node[inner sep = 0.3mm] (r) at (9.25,0.3) {\small $\underline{r}_110\underline{r}_2 01\underline{r}_3$};
  \node[inner sep = 0.3mm] (t) at (7.75,1.2) {\small $\underline{r}_110\underline{r}_2 10\underline{r}_3$};
  \foreach \from/\to in {b/l,b/r,t/l,t/r} \draw(\from)--(\to);
  \end{tikzpicture}
 \caption{Intervals of height $2.$ } 
 \label{fig:interval}
\end{figure}
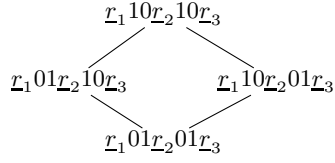
}
where $\underline{r}_1 10\underline{r}_2 10\underline{r}_3$ is an $(n,r)$-sequence.
We associate with $I$ the linear combination $\mathsf{sz}(I)$ defined by  
\[
\mathsf{sz}(I) = 
[\underline{r}_110\underline{r}_2 10\underline{r}_3]-[\underline{r}_1 10\underline{r}_2  01\underline{r}_3]
-[\underline{r}_1 01 \underline{r}_2 10 \underline{r}_3]+ [\underline{r}_1  01\underline{r}_2  01\underline{r}_3].    
\]

\begin{lemma} \label{kersp}  
For every height-$2$ interval $I$ in $(\mathcal{S}(n,r),\trianglerighteq),$ the linear combination $\mathsf{sz}(I)$ 
is in the kernel of $\mathsf{Sp}.$   
\end{lemma}
\begin{proof} 
By Lemma \ref{cornerstone},  
the differences
$\mathsf{Sp} ( [\underline{r} 10 \underline{s}] - [\underline{r} 01 \underline{s}])$ on ``opposite'' sides of the height-$2$ interval $I$ cancel and $\mathsf{Sp}(\mathsf{sz}(I)) = 0.$     
\end{proof}

We note that since $(\mathcal{S}(n,r),\trianglerighteq)$ is distributive, the sublattice generated by the atoms in an interval of height $k$ is a Boolean algebra and give rise to a linear combination with $2^k$ terms analogous to the $4$-term linear combinations $\mathsf{sz}(I).$   The larger linear combinations are also in $\ker \mathsf{Sp}.$  An easy argument shows that they are linear combinations of $4$-term linear combinations $\mathsf{sz}(I).$    

Let $\mathcal{K}$ be the linear subspace in $\mathcal{G}(n,r)$ spanned by the linear combinations $\mathsf{sz}(I).$   If $\underline{r}$ is meet-reducible (that is, not meet-irreducible), then it  has two or more ascents and is the minimum of a height-$2$ interval.  
Hence, $[\underline{r}]$ can be written, modulo $\mathcal{K},$ as a linear combination with integer coefficients of symbols $[\underline{s}],$ where 
$\underline{s} \triangleright \underline{r}.$   Repeating this argument, we conclude that in the quotient $\mathcal{G}(n,r)/\mathcal{K},$ every symbol $[\underline{r}]$ can be written as an integral linear combination  of symbols of meet-irreducibles, that is,    
\[
[\underline{r}] = \sum_{\underline{m} \trianglerighteq \underline{r},\, \underline{m} \in \mathsf{M}}  \alpha_{\underline{m}}  [\underline{m}]  \qquad \mathrm{mod} \,\,\mathcal{K},   
\eqno(3.1)\]
where the coefficients $\alpha_{\underline{m}}$ are integers.   Working in the opposite way, one derives analogous assertions for join-reducibles and irreducibles.  

\begin{lemma} \label{spankersp}  
 The symbols $[\underline{m}],\, \underline{m} \in \mathsf{M},$ span the quotient $\mathcal{G}(n,r)/\mathcal{K}.$   The symbols, $[\underline{j}],\,\underline{j} \in \mathsf{J},$ span $\mathcal{G}(n,r)/\mathcal{K}.$   In particular, 
\[
\dim \mathcal{G}(n,r)/\mathcal{K} \,\,\leq\,\, r(n-r) + 1.  
\]
\end{lemma}

We end this section with another consequence of Lemma \ref{cornerstone}

\begin{prop}\label{xy}
Let $M$ and $N$ be $(n,r)$-matroids.  Then $x+y-xy$ divides the difference $T(M)-T(N)$ of their Tutte polynomials. 
\end{prop}

\begin{proof}
By Lemma \ref{cornerstone}, if $\underline{r}$ covers $\underline{s},$ or $\underline{r}$ is covered by $\underline{s}$ in $(\mathcal{S}(n,r),\trianglerighteq),$  
then $x+y-xy$ divides $\mathsf{Sp}([\underline{r}] - [\underline{s}]).$  Since the Hasse or covering diagram of   $(\mathcal{S}(n,r),\trianglerighteq)$ is connected, $x+y-xy$ divides $\mathsf{Sp}([\underline{r}] - [\underline{s}])$ for any pair of $(n,r)$-sequences.  Now observe that the difference $\mathcal{G}(M) - \mathcal{G}(N)$ is a sum of differences 
$[\underline{r}] - [\underline{s}].$  From this, we conclude that $T(M)-T(N)$ is a sum of differences 
$\mathsf{Sp}([\underline{r}] - [\underline{s}])$ and hence, $x+y-xy$ divides $T(M)-T(N).$   
\end{proof}

Proposition \ref{xy} implies that for points $(\alpha,\beta)$ on the curve $x+y-xy=0, $ the value of 
$  T(M;\alpha, \beta)  $ is constant on $(n,r)$-matroids $M.$     
In fact, if $M$ is an $(n,r)$-matroid and $\alpha \neq 1,$ $ T(M;\alpha, \beta) = \alpha^n  (\alpha -1)^{r-n}.$   

\section{Freedom matroids}  \label{freedom} 

Let $\underline{s}$ be an $(n,r)$-sequence and $b_1,b_2, \ldots,b_r$ be the positions where $1$'s occur in $\underline{s},$ arranged so that $b_1 < b_2 < \cdots < b_r.$ 
The {\em freedom matroid} $F(\underline{s})$ with {\sl defining sequence} $\underline{s}$  
is the matroid on the set $\{1,2,\ldots,n\}$   
in which 
\begin{enumerate} 
\item the elements $1,2, \ldots,b_1 - 1$ are loops (that is, in the closure $\mathrm{cl}(\emptyset)$), 
\item for $1 \leq j \leq r-1,$  $b_j$ is added as an isthmus and the elements $b_j, b_j + 1, b_j + 2, \ldots,b_{j+1}-1$ are freely positioned in  
$\mathrm{cl}(\{b_1,b_2, \ldots, b_{j}\}),$ and 
\item $b_r$ is added as an isthmus and $b_r,b_r+1, \ldots, n$ are freely positioned in the entire matroid. 
\end{enumerate}
The freedom matroid $F(\underline{s})$ has a distinguished flag (or maximal chain) of flats $X_0 \subset   X_1  \subset   \cdots   \subset   X_r,$  where $X_i = \{1,2,\ldots,b_{i+1} - 1\}$ for $0 \leq j \leq r-1,$ and 
$X_r=  \{1,2,\ldots,n\}.$  
Freedom matroids were first defined by Crapo in \cite{MR0190045};  they have been rediscovered many times and are also known 
as nested, counting, or Schubert matroids.  

The following lemma is immediate from the definition.  

\begin{lemma}\label{basis} 

(a) Let $Y$ be a rank-$i$ flat in $ F(\underline{s}).$  Then $|Y| \leq |X_i|.$   

\noindent
(b) A set $B$ in $\{1,2,\ldots,n\}$ is 
a basis of the freedom matroid $F(\underline{s})$ if and only if the $(n,r)$-sequence $\underline{r},$ defined by $r_j = 1$ if $j \in B,$ 
satisfies $\underline{s} \trianglerighteq \underline{r}.$ 
\end{lemma}
 
By Lemma \ref{basis}(b), if $\underline{s}_1 \trianglerighteq \underline{s}_2,$ then every basis of $F(\underline{s}_2)$
is a basis of $F(\underline{s}_1);$  in other words, $F(\underline{s}_1) \geq_w F(\underline{s}_2),$ where $\geq_w$ 
is the weak order on $(n,r)$-matroids.     

\begin{cor} 
The weak order on rank-$r$ freedom matroids on $\{1,2,\ldots,n\}$ is isomorphic to $(\mathcal{S}(n,r),\trianglerighteq).$ 
\end{cor}

\showhide{
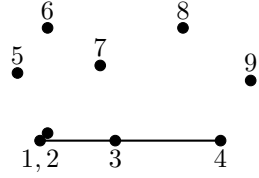
\begin{figure}
  \centering
  \begin{tikzpicture}[scale=1]
  \filldraw (-.3,.9) node[above] {$5$} circle  (2pt);
  \filldraw (0.1,1.5) node[above] {$6$} circle  (2pt);
  \filldraw (0.8,1.0) node[above] {$7$} circle  (2pt); \filldraw (1.9,1.5) node[above] {$8$} circle  (2pt); \filldraw (2.8,.8) node[above] {$9$} circle  (2pt);
  \filldraw (0,0) node[below] {$1,2$} circle  (2pt);
  \filldraw (0.1,0.1) node[below] {} circle  (2pt);
    \filldraw (1,0) node[below] {$3$\rule{0pt}{6.5pt}} circle  (2pt);
  \filldraw (2.4,0) node[below] {$4$} circle  (2pt);
  \draw[thick](0,0)--(2.4,0);
 \end{tikzpicture}
  \caption{The freedom matroid $F(101010000).$  }
  \label{fig:freedommatroid}
\end{figure}
}

We turn now to $\mathcal{G}$-invariants of freedom matroids.    

\begin{lemma}\label{terms} 

(a)  If $g_{\underline{r}}(F(\underline{s})) \neq 0,$ then $\underline{r} \trianglerighteq \underline{s}.$   

\noindent
(b)  The coefficient $g_{\underline{s}} (F(\underline{s}))$ is non-zero.   
\end{lemma}  
\begin{proof}  
Let $\underline{r}$ be the rank sequence associated with the permutation $\pi$ and $b_1, b_2, \ldots, b_r$ be the elements of $\{j: r_j = 1\}$ arranged in increasing order.   
Then $\underline{r}$ defines a flag $Y_0 \subset Y_1 \subset \cdots \subset Y_r,$ where $Y_i = \cl(\{ \pi(1),\pi(2), \ldots, \pi(b_i) \}).$   By Lemma \ref{basis}(a),  $b_i \leq |Y_i| \leq |X_i|.$   This implies that $\underline{r} \trianglerighteq \underline{s}.$    

To prove part (b), note that  
\[
g_{\underline{s}} (F(\underline{s})) =  (b_1 - 1)! (b_2 - b_1)! (b_3 - b_2)!  \cdots (b_r - b_{r-1})! (n-b_r +1)! \Phi,
\] 
where $\Phi$ is the number of flags $Y_0 \subset Y_1 \subset \cdots \subset Y_r$ in $F(\underline{s})$ such that $|Y_i| = |X_i|$ for all $i.$
\end{proof}

It follows from Lemma \ref{terms} that the system of equations 
\[
\sum_{\underline{r}}  g_{\underline{r}}(F(\underline{s})) [\underline{r}] = \mathcal{G}(F(\underline{s}))  
\eqno(4.1)\]
is triangular with non-zero diagonal coefficients.  Thus, we can invert the system (with a triangular matrix) and write a symbol as a linear combination of $\mathcal{G}$-invariants of freedom matroids. 

For example, when $n=4$ and $r=2,$  the matrix $\big(g_{\underline{r}}(F(\underline{s}))\big)$ and its inverse are 
\[
\left(
\begin{array}{cccccc}
4   &   0   &   0   &   0  &   0   &   0 
\\
4   &   2   &   0   &   0  &   0   &   0
\\
4   &   4   &   6   &   0  &   0   &   0
\\
4   &   4   &   0   &   6  &   0   &   0  
\\
4   &   6   &   6   &   6  &   4   &   0
\\
4   &   8   &   12   &  12  &  20   &   24
\end{array}
\right), \, \left(
\begin{array}{cccccc}
\tfrac {1}{4}    &   0   &   0   &   0  &   0   &   0 
\\
-\tfrac {1}{2}    &   \tfrac {1}{2}    &   0   &   0  &   0   &   0
\\
\tfrac {1}{6}    &   -\tfrac {1}{3}    &   \tfrac {1}{6}    &   0  &   0   &   0
\\
\tfrac {1}{6}    &   -\tfrac {1}{3}    &   0   &   \tfrac {1}{6}   &   0   &   0  
\\
0                &   \tfrac {1}{4}    &   -\tfrac {1}{4}    &   -\tfrac {1}{4}   &   \tfrac {1}{4}    &   0
\\
-\tfrac {1}{24}    &  - \tfrac {1}{24}    &   \tfrac {1}{8}    &  \tfrac {1}{8}   &  -\tfrac {5}{24}    &   \tfrac {1}{24} \end{array}
\right),
\]
where the rows/columns are indexed by the symbols/freedom matroids with defining sequences  
$0011,0101,0110,1001,1010,1100.$

\begin{thm}\label{freedombasis}  
The $\mathcal{G}$-invariants $\mathcal{G}(F(\underline{r})),$  $\underline{r} \in \mathcal{S}(n,r),$ form a basis for 
the vector space $\mathcal{G}(n,r).$  The change-of-basis matrices between the symbol basis and the $\mathcal{G}$-invariant basis are triangular. 
In particular, the $\mathcal{G}$-invariants $\mathcal{G}(M),$ where $M$ is an $(n,r)$-matroid, span $\mathcal{G}(n,r)$ and $\mathcal{T}(n,r)$ equals the image of $\mathsf{Sp}.$   
\end{thm}


\begin{prop} \label{upper} 
If $\mathcal{G}(F(\underline{r}))$ occurs in the expansion of $\mathcal{G}(M)$ with non-zero coefficient, then $\underline{r} \in \mathrm{supp}(M).$   
\end{prop}

\begin{proof} 
We use the theory of incidence algebras on partially ordered sets (see \cite{Rota,EC1}).  Since $g_{\underline{r}}(F(\underline{s})) \neq 0$ only if $\underline{s} \trianglerighteq \underline{r},$  the entries of the matrix $\big(g_{\underline{r}}(F(\underline{s}))\big)$ form an incidence function on the partially ordered set $(\mathcal{S}(n,r),\trianglerighteq).$  Hence, by incidence-algebra theory, the entries of the inverse matrix, which is the change-of-basis matrix $\Delta$ from the symbol basis to the freedom-matroid basis, form an incidence function, that is, the $(\underline{s},\underline{r})$-entry of $\Delta$ is non-zero only if  $\underline{s} \trianglerighteq\underline{r}.$  

Let $\vec{M}$ be the vector $(g_{\underline{r}}(M))$ of coefficients of $\mathcal{G}(M)$ in the symbol basis,  Then $\Delta \vec{M}$ is the vector of coefficients of $\mathcal{G}(M)$  in the freedom-matroid basis.  As the entries of $\Delta$ form an incidence function, the coefficient of $\mathcal{G}(F(\underline{r}))$ is non-zero only if $\underline{r} \trianglerighteq \underline{s}$ for some sequence $\underline{s}$ in $\mathrm{supp}(M),$ and by Lemma \ref{support}, only if $\underline{r} \in \mathrm{supp}(M).$  
\end{proof}



The linear map $\mathsf{Sp}$ sends $\mathcal{G}(M)$ to $T(M).$  Since $\mathcal{G}(M)$ can be written (uniquely) as a linear combination of $\mathcal{G}$-invariants of freedom matroids, we have the following corollary.  

\begin{cor}\label{TPsubspace} 
The Tutte polynomial of an $(n,r)$-matroid is a linear combination (not necessarily unique) 
of the Tutte polynomials 
$T(F(\underline{r})),$ $\underline{r} \in \mathcal{S}(n,r)$  of freedom matroids.     
\end{cor}

We end with a deletion-contraction lemma.  

\begin{lemma} \label{cdlemma}
Let $\underline{r}_1 1\check{0} \underline{r}_2$ be an $(n,r)$-sequence with a descent (indicated by a $\,\check{\,}$) at position $i.$    Then 
\[
T(F(\underline{r}_1 10 \underline{r}_2)) = T(F(\underline{r}_1 1 \underline{r}_2)) +  
T(F(\underline{r}_1 0 \underline{r}_2)). 
\]
\end{lemma}
\begin{proof}
Since a descent occurs at $i,$ the element $i$ is neither a loop nor an isthmus.  Hence, by the deletion-contraction recursion for Tutte polynomials,
\[
T(F(\underline{r}_1 10 \underline{r}_2)) = 
T(F(\underline{r})) \backslash i + T(F(\underline{r})) / i = 
T(F(\underline{r}_1 1 \underline{r}_2)) +  
T(F(\underline{r}_1 0 \underline{r}_2)). 
\qedhere  \]  
\end{proof}

\section{Two bases for $\mathcal{T}(n,r)$}  \label{Twobases}

In this section, we find two bases for $\mathcal{T}(n,r),$ one coming from the join-irreducibles, the other from the meet-irreducibles of $(\mathcal{S}(n,r),\trianglerighteq).$
We begin with the basis built from the join-irreducibles.  

\begin{thm}  
The Tutte polynomials $T(F(0^a1^b0^c1^d)),$ where $a+b+c+d = n$ and $b+d=r,$  are linearly independent in $\mathcal{T}(n,r).$   
\end{thm}

We begin the proof with two formulas.  The freedom matroid $F(1^r 0^{n-r})$ is the uniform matroid $U_{r,n}.$  
Writing $r=b$ and $n-r=c,$ its Tutte polynomial is given by 
\[
T(U_{b,b+c}) = \sum_{j=0}^{b-1} \binom {c-1+j}{j}x^{b-j}  + \sum_{k=0}^{c-1} \binom {b-1+k}{k} y^{c-k}.     
\eqno(5.1)\]
For example, 
\[
T(U_{5,9})= x^5 + 4 x^4 + 10 x^3 + 20 x^2 +35x + 35 y + 15 y^2 + 5 y^3 + y^4.  
\]
The freedom matroids $F(0^a1^b0^c1^d)$ are direct sums of the uniform matroid $U_{b,b+c}$ with $a$ loops and $d$ isthmuses and hence, 
\[
T(F(0^a1^b0^c1^d)) = y^a T(U_{b,b+c})x^d.   
\eqno(5.2)\] 
For example, 
\[
T(F(0^2 1^3 0^2  1^3)) = y^2x^6 + 2y^2x^5 + 3y^2x^4 + 3y^3x^3 + y^4 x^3.     
\]

Using these formulas, we write down the matrix $\Gamma$ of coefficients of Tutte polynomials $T(F(0^a1^b0^c1^d)).$ The columns of $\Gamma$ are indexed by freedom matroids with defining sequences  
\begin{eqnarray*}
&&
1^r0^{n-r}, 1^{r-1}0^{n-r}1, \ldots, 1^2 0^{n-r}1^{n-r-2}, 10^{n-r}1^{n-r-1}, \,\, 
\\
&&
01^r0^{n-r-1}, 01^{r-1}0^{n-r-1}1, \ldots, 01^2 0^{n-r-1}1^{r-2}, 010^{n-r-1}1^{r-1}, \,\,
\\
&& \qquad \vdots 
\\
&&
0^{n-r-1} 1^r0, 0^{n-r-1} 1^{r-1}01, \ldots, 0^{n-r-1} 1^2 01^{n-r-2}, 0^{n-r-1} 101^{r-1},\,\,
\\
&& 0^{n-r}1^r 
\end{eqnarray*} 
in the given order and the rows are indexed by monomials  
\begin{eqnarray*}
&&  
x^r,x^{r-1},\ldots,x^2,x, \,\,   
\\
&&
yx^r, yx^{r-1},\ldots, yx^2, yx,   \,\, 
\\
&&
y^2x^r,y^2 x^{r-1},\ldots,y^2x^2, y^2x, \,\,
\\
&& \qquad \vdots 
\\
&& 
y^{n-r-1}x^r,  y^{n-r-1}x^{r-1},\ldots,y^{n-r-1}x^2,y^{n-r-1}x, \,\,
\\
&& y^{n-r}x^r,    
\end{eqnarray*} 
in the given order, 
with the remaining monomials following in any order.   
Note that the columns and the first $r(n-r)+1$ rows are divided into $n-r$ {\em blocks}, each with $r$ indices, and one additional index.   We label the blocks on the rows by the number of $0$'s at the beginning of the defining sequence and the blocks on the columns by the exponent of the variable $y;$ in both cases, the label ranges from $0$ to $n-r-1.$  
For example, when $r=3$ and $n=5,$ there are two blocks, each of size $3,$ and $\Gamma$ is the $11 \times 7$ matrix  
\[
\begin{array}{cccccccc}
\,    &11100  &  11001  &  10011  &  01110  & 01101  & 01011  & 00011
\\
x^3 &  1  &   1      &   1    &     0    &        0          &        0     &     0        
\\
x^2 &  2  &    2     &    0   &      0   &        0          &          0   &       0      
\\
x &    3  &     0      &     0   &    0      &         0         &         0    &            0 
\\
yx^3 &  0   &   0    &    0    &    1   &           1        &       1      &        0         
\\
yx^2 &  0   &     0  &     1  &     1   &           1       &        0     &       0         
\\
yx &   0  &   2        &     0   &      1    &         0        &        0     &         0       
\\
y^2x^3  & 0   & 0  &    0    &     0   &           0        &       0      &        1       
\\
y  & 3   &      0       &     0    &       0  &            0       &       0      &          0      
\\
y^2  &  1   &    0    &    0   &     1   &           0        &        0     &      0         
\\
y^2x^2  &   0  &  0  &   1    &    0   &          0         &       1     &  0             
\\
y^2x  &  0 &   1      &    0   &   0     &           1         &        0     &           0     
\end{array}
\]

It follows from formulas (5.1) and (5.2) that in the Tutte polynomial $T(F(0^a 1^b 0^c 1^d)),$ all the monomial in blocks $0$ to $a-1$ have zero coefficient, and in block $a,$ only the monomials $y^a x^r, y^ax^{r-1}, \ldots, y^a x^{r-b+1}$ have  non-zero coefficients.   In addition, the last Tutte polynomial, $T(F(0^{n-r}1^r)),$ equals $y^{n-r}x^r$ and 
the monomial $y^{n-r}x^r$ does not occur (with non-zero coefficient) in any other 
Tutte polynomials $T(F(0^a 1^b 0^c 1^d)).$  

\begin{lemma}\label{blockmatrices}
The block diagonal submatrix $U_i,$   with rows and columns indexed by the $i$th block,  
are upper triangular matrices with non-zero diagonal entries.  The upper block submatrices, with rows indexed by the $i$th block and  columns indexed by the $j$th block, with $i < j,$   
are zero matrices.  
The only non-zero entry in the $(r(n-r)+1)$st row and $(r(n-r)+1)$st column is the diagonal entry.  
\end{lemma}  

Lemma \ref{blockmatrices} implies that the matrix $\Gamma$ restricted to the first $r(n-r)+1$ rows has the form 
\[
\left(\begin{array}{cccccc}
U_0    &    0   &   0   &  \ldots  &  0  &  0
\\
*        & U_1  &   0   &             &  0  &  0
\\
*        &   *    &   U_2 &           &  0  &  0 
\\
          &&&                     \ddots  &&   \vdots
\\
*       &*      &*        &             & U_{n-r-1}  &  0  
\\
0       &   0     & 0         &  \ldots             &   0                 &   1
\end{array}\right),
\]
where the block diagonal submatrices $U_i$ are $r \times r$ upper triangular matrices and the asterisks $*$ are $r \times r$ matrices.
From this, it is evident that the matrix $\Gamma$ has rank $r(n-r)+1.$  
We conclude that 
\[
\dim \mathcal{T}(n,r) \geq r(n-r)+1.  
\]   
Combining this with Corollary \ref{spankersp} and the first homomorphism theorem for vector spaces, we obtain 
\[
r(n-r)+1 \geq \dim \mathcal{G}(n,r)/\mathcal{K} \geq \dim \mathcal{T}(n,r) \geq r(n-r)+1, 
\]
and hence $\mathcal{K} = \ker \mathsf{Sp}$ and 
\[
 \mathcal{G}(n,r)/\mathcal{K} \cong \mathcal{T}(n,r).
\] 
The following results are immediate consequences.    

\begin{thm} \label{main1}
The linear combinations $\mathsf{sz}(I),$ where $I$ ranges over all height-$2$ intervals of $(\mathcal{S}(n,r),\trianglerighteq),$ span $\ker \mathsf{Sp}.$ 
The symbols $[\underline{j}], \, \underline{j} \in \mathsf{J},$ form a basis for the quotient $\mathcal{G}(n,r)/\ker \mathsf{Sp}$ and the symbols $[\underline{m}], \, \underline{m} \in \mathsf{M},$ form a basis for $\mathcal{G}(n,r)/\ker \mathsf{Sp}.$    
\end{thm}

\begin{thm}\label{mainjoin}
The Tutte polynomials $T(F(\underline{j})), \, \underline{j} \in \mathsf{J},$ form a basis for $\mathcal{T}(n,r).$  In particular, 
\[
\dim \mathcal{T}(n,r) =r(n-r)+1. 
\]  
\end{thm}

\noindent 
We call $\{T(F(\underline{j})): \underline{j} \in \mathsf{J} \}$ the {\sl  join-irreducible basis} of $\mathcal{T}(n,r).$

We will now consider the basis built from the meet-irreducibles, that is, $(n,r)$-sequences of the form $1^a0^b1^c0^d.$   When $c =1,$ $F(1^{r-1}0^b10^{n-r-b})$ is the paving matroid on $\{1,2,\ldots,n\}$ with one non-trivial copoint $\{1,2,\ldots, r-1+b\}.$    
Freedom matroids defined by meet-irreducibles can be characterized by their cyclic flats. Recall that a set is {\sl cyclic} if it is a union of circuits.  

\begin{lemma}   Let $n > r.$ 
An $(n,r)$-matroid $M$ is isomorphic to a freedom matroid $F(\underline{s}),$ where $\underline{s}$ is a meet-irreducible, if and only if $M$ contains exactly one cyclic flat, or exactly two cyclic flats, one of which is the entire set $\{1,2,\ldots,n\}.$    
\end{lemma}

\begin{proof}  
If $a = r,$ then $c=0$ and $F(1^a 0^b 1^c 0^d) \cong U_{r,n}.$ If $a < r$ and $d = 0,$ then 
$F(1^a 0^b 1^c 0^d)$ is the direct sum of the uniform matroid $U_{a,a+b}$ and $r-a$ isthmuses.   In the general case, when $a < r$ and $d > 0,$ $F(1^a 0^b 1^c 0^d)$ has two cyclic flats, $\{1,2,\ldots,a+b\}$ and $\{1,2,\ldots,n\}.$   
This argument can be reversed and the lemma follows.  
\end{proof} 

In contrast to the Tutte polynomials in the join-irreducible basis, the Tutte polynomials $T(F(1^a0^b1^c0^d))$ have  complicated formulas.  These formulas will be described in Section \ref{mi}.  The next lemma gives a formula for the case $b=c=1.$           

\begin{lemma}\label{TPPaving}  Let $r \geq 1$ and $n \geq 2.$    Then 
\[
T(F(1^{r-1} 0 1 0^{n-r-1})) =  T(F(1^r  0^{n-r}))- (x + y -xy).  
\]
\end{lemma} 


\begin{proof}
A simple calculation yields 
\begin{eqnarray*} 
\mathcal{G}(F(1^r0^{n-r})) &=& \binom {n}{r} r! (n-r)! [1^r0^{n-r}], 
\\
\mathcal{G}(F(1^{r-1}010^{n-r-1})) &=& \left( \binom {n}{r}-1 \right)  r! (n-r)! [1^r0^{n-r}], 
+ r!(n-r)![  1^{r-1}010^{n-r-1}  ] 
\end{eqnarray*}
and hence 
\[
\mathcal{G}(F(1^r0^{n-r})) - 
\mathcal{G}(F(1^{r-1}010^{n-r-1})) = r! (n-r)!( [1^r0^{n-r}] -  [  1^{r-1}010^{n-r-1}  ]).
\]
We can now finish the proof by applying $\mathsf{Sp}$ and  Lemma \ref{cornerstone}. 
\end{proof}

Lemma \ref{TPPaving} can also be proved using induction and deletion-contraction.   The method of proof yields a more general result, which we state without a proof.    
(For the definition of circuit-hyperplane relaxation, see \cite[p.~39]{ox}.)  

\begin{prop} \label{circuithyperplane} 
Let $M^{\prime}$ be obtained from $M$ by a circuit-hyperplane relaxation.  Then 
\begin{eqnarray*} 
\mathcal{G}(M^{\prime}) - 
\mathcal{G}(M) &=&   r! (n-r)!( [1^r0^{n-r}] -  [  1^{r-1}010^{n-r-1}  ]),  
\\
T(M^{\prime})-T(M) &=& x+y-xy. 
\end{eqnarray*} 
\end{prop}

To show that $\{T(F(\underline{m})): \, \underline{m} \in \mathsf{M}\}$ is a basis, we first show that 
the $\mathcal{G}$-invariants $\mathcal{G}(F(\underline{m})),\, \underline{m} \in \mathsf{M}$ form a basis of the quotient vector space $\mathcal{G}(n,r)/\ker \mathsf{Sp}.$     To see this, 
we use equations (3.1) and (4.1) to obtain, for $\underline{m} \in \mathsf{M},$  
\begin{eqnarray*} 
\mathcal{G}(F(\underline{m})) &=& \sum_{\underline{r} \trianglerighteq \underline{m}}  g_{\underline{r}}(F(\underline{m})) 
[\underline{r}]   
\\
&= &  
\sum_{\underline{r} \trianglerighteq \underline{m}}  g_{\underline{r}}(F(\underline{m}))   
\left(  \sum_{\underline{m}^{\prime} \trianglerighteq \underline{r}}  
\alpha_{\underline{m}^{\prime}} [\underline{m}^{\prime}] \right)   
\\
& = & 
\sum_{\underline{m}^{\prime} \trianglerighteq \underline{m}}  h_{\underline{m}^{\prime}}(F(\underline{m}))   
[\underline{m}^{\prime}]     
\end{eqnarray*}
in the quotient $\mathcal{G}(n,r)/\ker \mathsf{Sp}.$  Note that $h_{\underline{m}}(F(\underline{m})) = g_{\underline{m}}(F(\underline{m}))$ and hence $h_{\underline{m}}(F(\underline{m})) \neq 0.$    
Applying $\mathsf{Sp},$ we obtain 
\[
\sum_{\underline{m}^{\prime} \trianglerighteq \underline{m}}  h_{\underline{m}^{\prime}}(F(\underline{m}))  \mathsf{Sp} [\underline{m}^{\prime}]= T(F(\underline{m})),   
\eqno(5.3)\] 
a triangular system of equations with non-zero diagonal coefficients relating the sets $\{T(F(\underline{m})): \, \underline{m} \in \mathsf{M}\}$ and $\{\mathsf{Sp}[\underline{m}]:\, \underline{m} \in \mathsf{M}\}$ in $\mathcal{T}(n,r).$  
However, Theorem \ref{main1} implies that $\{\mathsf{Sp} [\underline{m}]:\, \underline{m} \in \mathsf{M}\}$ is a basis for $\mathcal{T}(n,r).$   Hence, we obtain the following theorem.      

\begin{thm}\label{main2}
The Tutte polynomials $T(F(\underline{m})),$ where $\underline{m} \in \mathsf{M},$ form a basis for $\mathcal{T}(n,r).$    
\end{thm}

\noindent 
We  call $\{ T(F(\underline{m})):\, \underline{m} \in \mathsf{M}\}$ the {\sl meet-irreducible basis} of 
$\mathcal{T}(n,r).$

\section{Building linear relations} \label{LR}

In hindsight, the key to finding linear relations on Tutte polynomials of freedom matroids is Lemma \ref{TPPaving}. This lemma says that the difference $T(F(1^r0^{n-r})) - T(F(1^{r-1}010^{n-r-1}))$ equals $x + y - xy,$ a polynomial not depending on $r$ and $n.$    
We will show that similar assertions hold for linear combinations derived from height-$2$ intervals in $(\mathcal{S}(n,r),\trianglerighteq).$  The next lemma gives the smallest case.     

\begin{lemma}\label{smallest} 
\[
T(F(1010)) - T(F(1001))
-  T(F(0110)) +  T(F(0101)) = x + y - xy.
\eqno(6.1)\]
\end{lemma} 
\begin{proof}  By Lemma \ref{cdlemma},      
\begin{eqnarray*} 
&& T(F(1010)) = T(F(101)) + T(F(100)),  
\\
&& T(F(1001)) = T(F(101)) + T(F(001)),  
\\
&&  T(F(0110)) = T(F(011)) + T(F(010)), 
\\
&&  
T(F(0101)) = T(F(011)) + T(F(001)).
\end{eqnarray*}
Hence, the left-hand side of equation (6.1) equals $T(F(100)) - T(F(010)),$  which in turn, equals 
$ (x + y + y^2) - (xy + y^2)  $ by direct computation.   
\end{proof}

We introduce a compact notation for $4$-term linear combinations.  Define 
\[
L(\underline{r}_1 \| \underline{r}_2 \| \underline{r}_3) =  T(F(\underline{r}_1 10 \underline{r}_2 10 \underline{r}_3))
-   T(F(\underline{r}_1 10 \underline{r}_2 01 \underline{r}_3))
-   T(F(\underline{r}_1 01 \underline{r}_2 10 \underline{r}_3)) + 
 T(F(\underline{r}_1 01 \underline{r}_2 01 \underline{r}_3)). 
\]

\begin{prop} \label{nodescent}   
\[
L (0^a1^b  \|  0^c1^d  \| 0^e 1^f )  = x^f y^a(x + y - xy).
\]
\end{prop} 
\begin{proof}  
We begin by observing that $F(0^a \underline{u} 1^f)$ is the direct sum of $F(\underline{u}),$ $a$ loops, and $f$ isthmuses 
and hence, 
\[
T(F(0^a \underline{u} 1^f)) = x^f y^a T(F(\underline{u}))   
\]
and 
\[
L (0^a1^b  \|  0^c1^d  \| 0^e 1^f )  = x^f y^a L (1^b  \|  0^c1^d  \| 0^e).
\]
Thus, it suffices to prove 
\[
L (1^b  \|  0^c1^d  \| 0^e)  = x + y - xy.
\]
We first show that value of $L (1^b  \|  0^c1^d  \| 0^e)$ is independent of $d.$  To do this, we use Lemma \ref{cdlemma} at the descent indicated by a $\,\check{\,}\,$ to obtain the 
deletion-contraction identities  
\begin{eqnarray*} 
T(F(1^b  10 0^c1^d 1\check{0}  0^e)) &=& T( F(1^b 10^{c+1}1^{d+1} 0^e)) + T( F(1^{b} 10  0^c 1^{d-1} 10 0^e)), 
\\
T(F(1^b  10 0^c1^d \check{0} 1  0^e)) &=& T( F(1^b10^{c+1}1^{d+1}0^e)) + T(F(1^b  10 0^c 1^{d-1} 01 0^e)), 
\\
T(F(1^b 01 0^c1^d 1\check{0} 0^e)) &=& T(F(1^b 01 0^{c}1^{d+1} 0^e)) + T(F(1^b 01 0^c 1^{d-1} 10 0^e)), 
\\
T(F(1^b 01 0^c1^d \check{0}1 0^e)) &=& T(F(1^b 010^{c}1^{d+1} 0^e)) + T(F(1^b 01 0^c 1^{d-1} 01 0^e)).   
\end{eqnarray*} 
These identities imply that  
\[
L (1^b  \|  0^c1^d  \| 0^e) = L (1^b  \|  0^c1^{d-1}  \| 0^e).  
\]
By induction, 
\[
{L} (1^b  \|  0^c1^d  \| 0^e) = {L} (1^b  \|  0^c \| 0^e).
\]

We deal next with $c.$   We use the deletion-contraction recursions (at the element indicated by a $\check{\,}$)  
\begin{eqnarray*} 
T(F(1^b 10 0^c 1\check{0} 0^e)) &=& T(F(1^b 10 0^{c} 1 0^e)) + T(F(1^{b+1}  0^{c+e+2})), 
\\
T(F(1^b 1\check{0} 0^c 01 0^e)) &=& T(F(1^b 1 0^{c} 0 1 0^e)) + T(F(1^{b} 0^{c+2} 1 0^e)), 
\\
T(F(1^b 01 0^c 1\check{0} 0^e)) &=& T(F(1^b 01 0^{c}1 0^e)) + T(F(1^b 01 0^{c+e+1})), 
\\
T(F(1^b 01 \check{0}0^{c-1} 01 0^e)) &=& T(F(1^b 010^{c}1 0^e)) + T(F(1^b 0^{c+2} 1 0^e)). 
\end{eqnarray*}
to conclude that 
\[
{L} (1^b  \|  0^c \| 0^e) =  
T(F(1^{b+1} 0^{c+ e +2})) - T(F(1^b 01 0^{c+e+1})).
\]
By Lemma \ref{TPPaving},  
\[
L (1^b  \|  0^c \| 0^e) =  x + y -xy,  
\]
completing the proof.  
\end{proof}

Noting that a bit sequence has no descent if and only if it equals $0^a 1^b,$  Proposition \ref{nodescent} allows us to calculate    
$L (\underline{r}_1 \| \underline{r}_2 \| \underline{r}_3)$ when none of the bit sequences $\underline{r}_1,  \underline{r}_2, \underline{r}_3$ has a descent.   We shall describe a {\sl reduction} which writes $L (\underline{r}_1 \| \underline{r}_2 \| \underline{r}_3)$ as a linear combination of $4$-term linear combinations 
$L (\underline{s}_1 \| \underline{s}_2 \| \underline{s}_3),$ where $\underline{s}_1, \underline{s}_2,$ or $\underline{s}_3,$ have fewer descents. To keep track of the reductions, we use a binary tree.  

Let $\underline{u}$ be a bit sequence.  A {\sl descent tree} of $\underline{u}$ is a binary tree constructed in the following way.  Start with the sequence $\underline{u}.$  If $\underline{u}$ has a descent, say 
$\underline{u} = \underline{u}_1 10 \underline{u}_2,$ then add two descendants 
$\underline{u}_1 1 \underline{u}_2$ and $\underline{u}_1 0 \underline{u}_2$ to the node $\underline{u}.$  Continue for each bit sequence in the partially constructed tree until all the leaves are bit sequences with no descents.  
There are many descent trees, one for each ordering by which we choose the descents.

\showhide{
\begin{figure}
  \centering
  \begin{tikzpicture}[scale=1]
\node[inner sep = 0.3mm] (rll3) at (9.9,-0.3) {\small $01$}; 

\node[inner sep = 0.3mm] (rlr3) at (8.8,-0.3) {\small $00$};
\node[inner sep = 0.3mm] (rr2) at (11.1,0.4) {\small $011$};
\node[inner sep = 0.3mm] (rl2) at (9.3,0.4) {\small $0\underline{10}$};
\node[inner sep = 0.3mm] (lr2) at (6.9,0.4) {\small $001$};
\node[inner sep = 0.3mm] (ll1) at (5.1,0.4) {\small $000$};
  \node[inner sep = 0.3mm] (l1) at (6.1,1) {\small $00\underline{10}$};
  \node[inner sep = 0.3mm] (r1) at (9.9,1) {\small $01\underline{10}$};
  \node[inner sep = 0.3mm] (t) at (8.0,1.9) {\small $0\underline{10}10$};
  \foreach \from/\to in {t/l1,t/r1,ll1/l1,lr2/l1,rl2/r1,rr2/r1,rll3/rl2,rlr3/rl2} \draw(\from)--(\to);
  \end{tikzpicture}
 \caption{A descent tree of $01010;$  the leaves yield the Tutte polynomial $x^2y + xy +xy^2 +y^2 + y^3 $ of $F(10101).$} 
 \label{descenttree}
\end{figure}
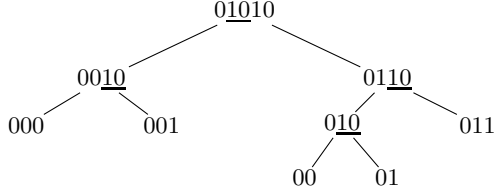
}

By Lemma \ref{cdlemma}, each branching of a descent tree gives a deletion-contraction decomposition of a freedom matroid.  Hence, a descent tree of $\underline{u}$ gives a decomposition of the freedom matroid $F(\underline{u})$   into a linear combination of
direct sums of loops and isthmuses in the Tutte-Grothendieck ring of matroids (see \cite[p.~243]{Brylawski} or \cite[ Section 6.2, p.~124]{BryOx}).  This yields the following lemma.  

\begin{lemma}  \label{leaves}   
The multiset of leaves of a descent tree of a bit sequence $\underline{u}$ depends only on $\underline{u}.$  Under the specialization $0^a 1^b \mapsto x^by^a,$ the sum over the multiset of leaves in a descent tree of $\underline{u}$  
equals the Tutte polynomial $T(F(\underline{u});x,y).$  
\end{lemma}

We define the following functions:  
\[ f(\underline{u}) = T(F(\underline{u});1,y),  
\,\,
g(\underline{u})  = T(F(\underline{u});x,1), 
\,\, 
\tau(\underline{u}) = T(F(\underline{u});1,1).   
\]

\begin{thm}\label{Linterval}   
$\,
L(\underline{r}_1 \| \underline{r}_2 \| \underline{r}_3) =  f(\underline{r}_1) \tau(\underline{r}_2)g(\underline{r}_3) (x + y - xy).
$
\end{thm}
\begin{proof} 
We use the reduction process.  Suppose $\underline{r}_1 = \underline{r}_{11} 10 \underline{r}_{12}.$  Then by deletion-contraction,   
\[
T(F(\underline{r}_1 01 \underline{r}_2 01 \underline{r}_3))  = T(F(\underline{r}_{11} 0 \underline{r}_{12}    01 \underline{r}_2 01 \underline{r}_3)) + 
T(F(\underline{r}_{11} 1 \underline{r}_{12}    01 \underline{u}_2 01 \underline{r}_3)).   
\]
Similar equations hold for the other three Tutte polynomials in the linear combination $L(\underline{r}_1 \| \underline{r}_2 \| \underline{r}_3)$ and combining these equations, we obtain 
\[
{L}(\underline{r}_1 \| \underline{r}_2 \| \underline{r}_3) ={L}(\underline{r}_{11} 0 \underline{r}_{12} \| \underline{r}_2 \| \underline{r}_3) + {L}(\underline{r}_{11} 1 \underline{r}_{12} \| \underline{r}_2 \| \underline{r}_3).
\]  
Iterating the reduction on $\underline{r}_1,$  and repeating the entire process on $\underline{r}_2$ and $\underline{r}_3,$ we obtain 
\[
{L}(\underline{r}_1 \| \underline{r}_2 \| \underline{r}_3) = \sum_{(\ell_1, \ell_2, \ell_3)}  {L}(\ell_1 \| \ell_2  \| \ell_3), 
\]
the sum ranging over all triples $(\ell_1,\ell_2,\ell_3),$ with $\ell_i$ a leaf in a descent tree of $\underline{r}_i.$  To finish, we apply Lemmas \ref{nodescent} and \ref{leaves}.  
\end{proof}

The next lemma generalizes Lemma \ref{TPPaving} and can be proved by the method in the proof of Theorem \ref{Linterval}.  

\begin{thm} \label{TPPaving1}
Let $a, b \geq 1.$   Then 
\[
T(F(\underline{r}_1  1^a0^b  \underline{r}_3))-
T(F(\underline{r}_1  1^{a-1} 01 0^{b-1} \underline{r}_3)) = f(\underline{r}_1) g(\underline{r}_3) (x + y -xy).
\]
\end{thm}

Theorems \ref{Linterval} and \ref{TPPaving1} give the building blocks for making linear relations on 
Tutte polynomials of freedom matroids.   The {\sl smallest} linear relation, holding in $\mathcal{T}(4,2),$ is 
\begin{eqnarray*}
 T(F(1100)) - T(F(1010)) &=& x + y -xy 
\\
& = &   T(F(1010)) -  T(F(1001)) - T(F(0110)) +T(F(0101)).  
\end{eqnarray*} 
A more complicated example is obtained by considering the two overlapping height-$2$ intervals in Figure \ref{fig:twodiamonds}.    The lower height-$2$ interval gives 
\[
L( \underline{r}_1 \| \underline{r}_2 0 \| \underline{r}_3) = f(\underline{r}_1)\tau(\underline{r}_2 0) g(\underline{r}_3)(x + y -xy) 
\]     
while the upper height-$2$ interval gives 
\[
L( \underline{r}_1 \| \underline{r}_2  \| 0\underline{r}_3) = f(\underline{r}_1)\tau(\underline{r}_2)  
g(0\underline{r}_3)(x + y -xy). 
\]     
Noting that $g(0\underline{r}_3) = g(\underline{r}_3),$ this yields 
\[
\tau   L( \underline{r}_1 \| \underline{r}_2 0 \| \underline{r}_3)  
= \tau^{\prime}   L( \underline{r}_1 \| \underline{r}_2  \| 0\underline{r}_3), 
\]
where $\tau = \tau(\underline{r}_2)$ and $\tau^{\prime} =  \tau(\underline{r}_2 0).$  Explicitly, 
\begin{eqnarray*} 
&& \tau T(F(\underline{r}_1 01 \underline{r}_2 001 \underline{r}_3))   
- (\tau + \tau^{\prime}) T(F(\underline{r}_1 01 \underline{r}_2  010 \underline{r}_3))  
+ \tau^{\prime}  T(F(\underline{r}_1 01 \underline{r}_2  100 \underline{r}_3))  
\\
&&  \quad -  \tau T(F(\underline{r}_1 10 \underline{r}_2 001 \underline{r}_3))   
+ (\tau + \tau^{\prime})  T(F(\underline{r}_1 10 \underline{r}_2  010 \underline{r}_3))  
- \tau^{\prime} T(F(\underline{r}_1 10 \underline{r}_2  100 \underline{r}_3))  
= 0.
\end{eqnarray*}

\showhide{
\begin{figure}
  \centering
  \begin{tikzpicture}[scale=1]
%
\node[inner sep = 0.3mm] (b) at (7.75,-0.5) {\small $\underline{r}_1 01 \underline{r}_2 001 \underline{r}_3$};
  \node[inner sep = 0.3mm] (l) at (6.5,0.3) {\small $\underline{r}_1 10 \underline{r}_2 001 \underline{r}_3$};
  \node[inner sep = 0.3mm] (r) at (9.5,0.3) {\small $\underline{r}_1 01 \underline{r}_2 010 \underline{r}_3$}; 
  \node[inner sep = 0.3mm] (r1) at (11.1,1.0) {\small $\underline{r}_1 01 \underline{r}_2 100 \underline{r}_3$};
  \node[inner sep = 0.3mm] (t) at (8.0,1.0) {\small $\underline{r}_1 10 \underline{r}_2  010 \underline{r}_3$}; 
  \node[inner sep = 0.3mm] (t1) at (9.7,1.8) {\small $\underline{r}_1 10 \underline{r}_2  100 \underline{r}_3$};
  \foreach \from/\to in {b/l,b/r,t/l,t/r,r/r1,t/t1,r1/t1} \draw(\from)--(\to);
  \end{tikzpicture}
  \caption{}
  \label{fig:twodiamonds}
\end{figure}
}

A generating set for all linear relations is the set $\mathsf{L}(n,r)$ consisting of the linear relations 
\begin{eqnarray*}  
&&  \tau(\underline{r}_2) [T(F(\underline{r}_1 1^{a+2} 0^{b+2} \underline{r}_3)) -
T(F(\underline{r}_1 1^{a+1} 010^{b+1} \underline{r}_3))]  = 
\\
&&  
\quad 
T(F(\underline{r}_1 10 \underline{r}_2 10 \underline{r}_3))
- T(F(\underline{r}_1 10 \underline{r}_2 01\underline{r}_3))
- T(F(\underline{r}_1 01 \underline{r}_2 10 \underline{r}_3))
 + T(F(\underline{r}_1 01 \underline{r}_2  01  \underline{r}_3)),
\end{eqnarray*}
where $\underline{r}_1,\underline{r}_2,\underline{r}_3$ are bit sequences such that 
$\underline{r}_1 01 \underline{r}_2 01 \underline{r}_2$ is an $(n,r)$-sequence,  $a$ is the number of $1$'s in $\underline{r}_2,$  and $b$ is the number of $0$'s in $\underline{r}_2.$   
The linear relations in $\mathsf{L}(n,r)$ are indexed by two intervals in $(\mathcal{S}(n,r),\trianglerighteq)$: the height-$2$ interval 
$[ \underline{r}_1 01 \underline{r}_2 01 \underline{r}_3, \underline{r}_1 10 \underline{r}_2 10\underline{r}_3]$ 
and lying above it, the height-$1$ interval $[\underline{r}_1 1^{a+1} 01 0^{b+1} \underline{r}_3, 
\underline{r}_1 1^{a+2} 0^{b+2} \underline{r}_3 ].$   The two intervals are disjoint or intersect at one sequence. 
As an example, the pair of intervals indexing the relation 
\[
T(F(11000)) - T(F(10100)) = T(F(10010)) - T(F(10001)) - T(F(01010)) + T(F(01001) )   
\] 
is shown in Figure \ref{fig:linrelation}.      

\begin{thm} \label{syzygies} 
The set $\mathsf{L}(n,r)$  is a generating or spanning set for linear relations or syzygies on Tutte polynomials of rank-$r$ size-$n$ freedom matroids.     
\end{thm} 
\begin{proof}
Propositions \ref{Linterval} and \ref{TPPaving1} imply that the linear relations in $\mathsf{L}(n,r)$ hold.  To show that $\mathsf{L}(n,r)$ is a generating set, note that if $\underline{s}$  is meet-reducible, then it is the minimum of a height-$2$ interval.  Using one of the linear relations in $\mathsf{L}(n,r),$ we can write 
$T(F(\underline{s}))$ as a linear combination with integer coefficients of Tutte polynomials $T(F(\underline{r })), \, \underline{r} \,\triangleright\, \underline{s}.$  Repeating this, we can write $T(F(\underline{r}))$ as an integral linear combination with integer coefficients   
of elements in the meet-irreducible basis.  We conclude that $\mathsf{L}(n,r)$ is a generating set.  
\end{proof} 

From the proof of Theorem \ref{syzygies}, we obtain the following proposition.  

\begin{prop}\label{integral}  
The Tutte polynomial of a freedom matroid $F(\underline{s}), \, \underline{s} \in \mathcal{S}(n,r),$ is a linear combination with integer coefficients of elements $T(F(\underline{m}))$ in the meet-irreducible basis such that $\underline{m} \trianglerighteq \underline{s}.$       
\end{prop}
%

%
The next proposition follows from Propositions \ref{upper} and \ref{integral}.  

\begin{prop}\label{straightening} 
Let $M$ be an $(n,r)$-matroid.   
Then the Tutte polynomial $T(M)$ is a linear combination of Tutte polynomials 
$T(F(\underline{m}))$ where $\underline{m} \in \mathsf{M} \cap \mathrm{supp}(M).$  
\end{prop}

\showhide{
\begin{figure}
  \centering
  \begin{tikzpicture}[scale=1]
  \node[inner sep = 0.3mm] (l) at (4.0,1.9) {\small $11000$};
  \node[inner sep = 0.3mm] (b) at (5.3,1.3) {\small $10100$};  
\foreach \from/\to in {b/l} \draw(\from)--(\to);  
\node[inner sep = 0.3mm] (b) at (7.75,-0.5) {\small $01001$};
  \node[inner sep = 0.3mm] (l) at (7,0.3) {\small $10001$};
  \node[inner sep = 0.3mm] (r) at (8.5,0.3) {\small $01010$};
  \node[inner sep = 0.3mm] (t) at (7.75,1.0) {\small $10010$};
  \foreach \from/\to in {b/l,b/r,t/l,t/r} \draw(\from)--(\to);
  \end{tikzpicture}
\caption{} 
 \label{fig:linrelation}
\end{figure}
}


The {\sl girth} (or {\sl spark}) of a matroid $M$ is the minimum size of a circuit in $M.$  
Let $\mathcal{T}_k(n,r)$ be the subspace of $\mathcal{T}(n,r)$ spanned by the Tutte polynomials of $(n,r)$-matroids of girth at least $k.$  

\begin{cor}\label{girth}   
The subspace $\mathcal{T}_k(n,r)$ has as basis the freedom matroids $T(F(\underline{m})), \, \underline{m} \in \mathsf{M}$ and $\underline{m}$ starts with $k-1$ $1$'s. 
In particular, 
\[ 
\dim \mathcal{T}_k(n,r) = (n-r)(r - k + 1)+1.    
\]
\end{cor} 
\begin{proof} 
Let $U$ be the interval $[1^{k-1} 0^{n-r} 1^{r-k+1}, 1^r 0^{n-r}]$ in $(\mathcal{S}(n,r), \trianglerighteq).$
It is immediate that the following are equivalent for an $(n,r)$-matroid $M$: 
\begin{enumerate} 
\item  $M$ has girth at least $k;$ 
\item  for all $(n,r)$-sequences $\underline{r},$  $g_{\underline{r}}(M) \neq 0$ implies that $\underline{r}$ starts with $k-1$ $1$'s;   
\item  $\mathrm{supp}(M) \subseteq U.$  
\end{enumerate}

The map that removes the initial segment $1^{k-1}$ from each bit sequence in $U$ is an order-isomorphism sending $U$ onto $(\mathcal{S}(n-k+1,r-k+1),\trianglerighteq).$  From the fact that a bit sequence is a meet-irreducible if and only if it equals $1^a0^b1^c0^d,$ 
it follows that the meet-irreducibles in $U$ map bijectively onto the meet-irreducibles in $\mathcal{S}(n-k+1,r-k+1)$ and hence, there are $(n-r)(r-k+1)+1$ such meet-irreducibles.    Since the meet-irreducibles in the {\it upper} interval $U$
are meet-irreducibles in $(\mathcal{S}(n,r),\trianglerighteq),$ it follows from Corollary \ref{straightening} that the Tutte polynomials $T(F(\underline{m})),$  where $\underline{m}$ is a meet-irreducible in $U,$ form an independent set spanning $\mathcal{T}_k(n,r).$   
\end{proof}

The third assertion in the next result uses the easy fact that a rank-$r$ matroid is paving if and only if it has girth 
$r$ or $r+1.$  

\begin{cor} 
The Tutte polynomials of loopless $(n,r)$-matroids span a subspace of dimension $(n-r)(r-1)+1,$  the Tutte polynomials of simple $(n,r)$-matroids span a subspace of dimension $(n-r)(r-2)+1,$ and 
the  Tutte polynomials of paving $(n,r)$-matroids span a subspace of dimension $n-r+1.$ 
\end{cor}


\section{Linear relations on coefficients of Tutte polynomials} \label{coefficients}

There are $(r+1)(n-r+1)$ possible non-zero coefficients $t_{ij},  \, 0 \leq i \leq r, 0 \leq j \leq n-r,$ in the Tutte polynomial of an $(n,r)$-matroid.  Since $\dim \mathcal{T}(n,r) = r(n-r)+1,$  there are $n$ linearly independent linear relations on the coefficients.  A set of such relations was found by Brylawski \cite[Section 6]{Brylawski}.   
For $m \geq 0,$ let 
\[
J_m = \sum_{\alpha=0}^{m} \left(\sum_{\beta=0}^{\alpha}  (-1)^{\beta}\binom {\alpha}{\beta} t_{m-\alpha,\beta}\right).  
\]
One might visualize the linear combinations $J_m$ as   ``tableaux with staircase shape''.  For example,    
\[
J_0 = t_{00}, \quad J_1 = \begin{array}{cc} \,\,\, t_{00} & -t_{01} \\ +\, t_{10} & \, \end{array}  \!,
\quad J_2 = \begin{array}{ccc} \,\,\,  t_{00} & -2t_{01} & + t_{02} \\  +\, t_{10} & -t_{11} & \, 
\\ +\, t_{20} &\, & \,\end{array}
\]
\vskip 0.1in
\[
J_3 = \begin{array}{cccc}
\,\,\,  t_{00}   & -3 t_{01}   & + 3 t_{02} & -  t_{03}  
\\
+\,  t_{10}   & - 2  t_{11}  &+  t_{12}     
\\
+ \,t_{20}  &  -   t_{21} 
\\
+\,  t_{30} 
\end{array}, 
\qquad 
J_4 = \begin{array}{ccccc}
\,\,\,  t_{00}   & -4  t_{01}  &+ 6 t_{02}   & -4  t_{03}  &+  t_{04}
\\
+\,  t_{10}   & -3 t_{11}   & + 3 t_{12} & -  t_{13}  &\, 
\\
+\,  t_{20}   & - 2  t_{21}  &+  t_{22}  &\,  &\,  
\\
+ \,t_{30}  &  -   t_{31} 
\\
+\,  t_{40} 
\end{array}.
\]
The {\sl hook} $H_m (d,a)$ is the linear combination in $J_m$ defined by 
\[
H_m(d,a)=
\sum_{j=1}^{m-d-a}   \binom {m-d-j}{a} t_{j+d,a}  + \binom {m-d}{a}  t_{da} + \sum_{k=1}^{m-d-a}  (-1)^k \binom{m-d}{a+k}  t_{d,a+k}.
\]
For example, 
\begin{eqnarray*} 
H_4(0,0) &=& t_{40} +  t_{30}  + t_{20} +  t_{10}  + t_{00} -4  t_{01} + 6 t_{02} -4  t_{03}+  t_{04},
\\
H_4(0,1) &=&    t_{31} + 2t_{21} + 3t_{11} +   4t_{01} - 6t_{02} + 4t_{03} - t_{04},   
\\
H_4(2,1) &=&      t_{31}  + 2  t_{21} - t_{22}.
\end{eqnarray*}  
The {\sl centered hook} $\bar{H}_m(d,a)$ is defined by subtracting $(d,a)$ from each pair of subscripts in $H_m(d,a),$ 
 that is,    
\[
\bar{H}_m(d,a)=
\sum_{j=1}^{m-d}   \binom {m-d-j}{a} t_{j0}  + \binom {m-d}{a}  t_{00} + \sum_{k=1}^{m-d}  (-1)^k \binom{m-d}{a+k}  t_{0,k}.
\]

The next theorem is Theorem 6.6 in Brylawski \cite{Brylawski}.  

\begin{thm}\label{sat}
If $n \geq 1$ and $0 \leq m < n,$ then the coefficients of the Tutte polynomial of a matroid on a set of size $n$ satisfy $J_m =0.$
\end{thm} 

 Brylawski proved Theorem \ref{sat} by showing that when $0 \leq m <n,$ $J_m$ is a linear combination of the number of flats (or closed sets) of corank $0$ and nullity strictly less than $n-r$  in an $(n,r)$-matroid, and because no such flats exist, $J_m = 0.$    This proof is somewhat complicated.   

We give another proof using the fact that $J_m = 0$ for all polynomials in $\mathcal{T}(n,r)$ if and only if $J_m =0$ for every polynomial in a basis of $\mathcal{T}(n,r).$    Thus, to prove Theorem \ref{sat}, it suffices to check $J_m =0$ for all  polynomials in the join-irreducible basis, that is, the Tutte polynomials $F(0^a 1^b 0^c 1^d).$ 
  
Recall that $T(F(0^a 1^b 0^c 1^d))=x^dy^a T(U_{b,b+c})$ and hence,     
\[
t_{jk}(F(0^a 1^b 0^c 1^d)) = t_{j-d,k-a}(U_{c,c+d}).
\]
Next, note that the coefficients in $t_{ij}(F(0^a1^b0^c1^d))$ are non-zero if and only if $i=d$ or $j=a$ but not both.  Hence, to check that these coefficients satisfy $J_m=0,$ it suffices to check that they satisfy $H_m(d,a)=0.$  But 
$T(F(0^a1^b0^c1^d))$ satisfy $H_m(d,a)=0$ if and only if 

\vskip 0.1in \hskip 0.5in  
($*$) \hskip 0.1in {\it    $T(U_{b,b+c})$ satisfy the centered hook relation $\bar{H}_m(d,a)=0.$  }        
\vskip 0.1in

\noindent
To finish our proof, we will prove assertion ($*$) by induction.  We start the induction by checking that ($*$) holds for the Tutte polynomials $T(U_{c,c+1})$ and $T(U_{c,c+2}).$   This is an easy calculation and we omit the details.  For the induction step, we note that by deletion-contraction,  
\[
T(U_{b,b+c}) = T(U_{b,b+c-1}) + T(U_{b-1,b-1+c}).
\]
By the induction hypothesis, the two Tutte polynomials on the right satisfy $\bar{H}_m(d,a)=0$ and by linearity, $T(U_{b,b+c})$ also satisfies $\bar{H}_m(d,a)=0.$    This completes our proof of Theorem \ref{sat}.     

As $J_{k}$ involves  coefficients, such as $t_{k0},$ not in $J_{k-1},$ the linear relations $J_m=0$ are linearly independent.   

\begin{cor}
For $n \geq 1,$ the linear relations $J_m = 0, m = 0,1,\ldots,n-1,$ give a basis for all linear relations on coefficients of Tutte polynomials of $(n,r)$-matroids. 
\end{cor}

\section{Formulas for Tutte polynomials of meet-irreducible sequences}\label{mi}  

In this section, we calculate the Tutte polynomials $T(F(1^a 0^b 1^c 0^d)).$

\begin{thm} \label{MI} When $a+c=r$ and $b+c = n-r,$ then the Tutte polynomial $T(F(1^a 0^b 1^c 0^d))$ equals 
\[
T(U_{r,n} )  \,-\, (x+y-xy)\sum_{j=0}^{b-1}  \binom {a+j}{j} y^{j}T^{\circ}(U_{c,c+b-j+d}),
\]
where 
\begin{eqnarray*}
T^{\circ}(U_{c,c+d}) = T^{\circ}(U_{c,c+d};x) &=& \frac {T(U_{c,c+d}; x,0)}{x}
\\
&=&   \sum_{k=0}^{c-1}  \binom{d-1+k}{j} x^{c-1-k}. 
\end{eqnarray*}
\end{thm}

\begin{proof} Let  $D(a,b,c,d) = T(F(1^{a+c}0^{b+d})) - T(F(1^a 0^b 1^c 0^d)).  $

\begin{lemma}  \label{recurseD} 
The differences $D(a,b,c,d)$ satisfy the boundary condition  
\[
D(a,0,c,d) = 0 
\]
and the recursion 
\[
D(a,b,c,d) = D(a-1,b,c,d) + D(a,b-1,c,d).
\]
\end{lemma}

\begin{proof}  The boundary condition is immediate from the definition.  To prove the recursion, note that by Lemma \ref{cdlemma}, 
\begin{eqnarray*}
T(F(1^{a+c}0^{b+d}))  &=& T(F(1^{a-1+c}0^{b+d})) + T(F(1^{a+c}0^{b-1+d})),
\\
T(F(1^a 0^b 1^c 0^d)) &=&  T(F(1^{a-1} 0^b 1^c 0^d)) + T(F(1^a 0^{b-1}  1^c 0^d)). 
\end{eqnarray*}  
\end{proof}

We now apply Lemma \ref{recurseD}, bearing in mind that the recursion for $D(a,b,c,d)$ on $a,b$ is the same as the  recursion for the binomial coefficients $\binom {a}{b},$  to obtain the following lemma.    

\begin{lemma} \label{formulaD} \hskip 0.2in
$ 
\displaystyle{
D(a,b,c,d) = \sum_{i=0}^{b-1}  \binom {a-1+i}{i} D(0,b-i,c,d).
}
$ 
\end{lemma}

To use Lemma \ref{formulaD}, we need an explicit formula for $D(0,b,c,d).$   This can be done using formulas for Tutte polynomials of uniform matroids.  

\begin{lemma}  \label{8.2}\hskip 0.1in
$\displaystyle{
D(0,b,c,d) =  (x+y-xy) \sum_{j=0}^{b-1} y^j  T^{\circ}(U_{c,c+b-j+d}),
}$ 

\noindent
where 
\begin{eqnarray*}
T^{\circ}(U_{c,c+d}) = T^{\circ}(U_{c,c+d};x) &=& \frac {T(U_{c,c+d}; x,0)}{x}
\\
&=&   \sum_{k=0}^{c-1}  \binom{d-1+k}{k} x^{c-1-k}. 
\end{eqnarray*}
\end{lemma}

\begin{proof}  We use the fact that 
\begin{eqnarray*}
D(0,b,c,d) 
&=& T(F(1^c 0^{b+d})) - T(F(0^b1^c0^d))
\\
&=& T(U_{c,c+ b+d})  - y^b  T(U_{c,c+d}) .
\end{eqnarray*} 
Thus, the case $b=1$ is 
\[
T(U_{c,c+d} ) - yT(U_{c,c+d-1}) = (x+y-xy) T^{\circ}(U_{c,c+d};x).
\eqno(8.1)\]
This formula can be proved by a simple calculation using Pascal's identity for binomial coefficients.  
Iterating formula (8.1) and ``telescoping'' yield Lemma \ref{8.2}. 
For example, 
\begin{eqnarray*}
T(U_{4,9}) - yT(U_{4,8}) 
&=& x^4 + 5x^3 +15x^2  + 35x + 35y + 20y^2 +  10y^3 +  4y^4  + y^5
\\
&& - (x^4y + 4x^3y +10x^2y  + 20xy + 20y^2 + 10y^3 +  4y^4 +  y^5)  
\\
&=&  (x+y -xy)(x^3 + 5x^2 + 15x + 35) 
\\
&=&  (x+y-xy)T^{\circ}(U_{4,9})
\end{eqnarray*}
 and    
\begin{eqnarray*}
&&  T(U_{4,9}) - y^3 T(U_{4,6}) 
\\
&=&
\big(T(U_{4,9}) - yT(U_{4,8})\big) + \big( yT(U_{4,8}) - y^2 T(U_{4,7}) \big) + \big( y^2 T(U_{4,7}) - y^3 T(U_{4,6}) \big) 
\\
&=& 
(x+y -xy)(T^{\circ}(U_{4,9}) +  yT^{\circ}(U_{4,8}) + y^2T^{\circ}(U_{4,7})).    
\end{eqnarray*}
\end{proof}
We can now finish the proof of Theorem \ref{MI} by combining the two lemmas, changing the order of summation, and using an elementary binomial-coefficient identity:  
\begin{eqnarray*}  
D(a,b,c,d) &=& \sum_{i=0}^{b-1}  \binom {a-1+i}{i}  \left( (x+y-xy) \sum_{j=0}^{b-1} y^j  T^{\circ}(U_{c,c+b-j+d}) \right)
\\
&=&  
(x+y-xy) \sum_{j=0}^{b-1} \binom {a+j}{j} y^j  T^{\circ}(U_{c,c+b-j+d}).
\end{eqnarray*} 
\end{proof}

We note the following special case of Theorem \ref{MI}. 

\begin{cor} \label{MIPaving}   
\[
T(F(1^{r-1}0^b10^{n-r-1})) = T(U_{r,n-r}) - (x+y-xy)\sum_{j=0}^{b-1}  \binom {r-1+j}{j}y^{b-1-j}.
\]
\end{cor}

From Corollary \ref{MIPaving}, we can easily derive a formula for the Tutte polynomial of a paving $(n,r)$-matroid.  

\begin{prop} \label{Paving 8.6}  
Let $P$ be a paving $(n,r)$-matroid with $f(s)$ copoints of size $s.$  Then 
\[
T(P) = T(U_{r,n}) - (x+y-xy)  \sum_{b \geq 1}  f(r-1+b)  \left(\sum_{j=0}^{b-1}  \binom {r-1+j}{j}y^{b-1-j} \right).
\]
\end{prop}

\begin{prop}\label{TPgirth}   
Let $M$ be an $(n,r)$-matroid having girth at least $k.$ Then $t_{\alpha \beta}(M) = 0$ unless 
\[
(\alpha,\beta) \in \{(k,0):  1  \leq k \leq r\} \cup \{(i,j): 0 \leq i\leq r-k+1, 1 \leq j \leq n-r\}.  
\] 
In particular, the linear relations $J_m=0, \, 0 \leq m \leq n-1,$  and 
\[
t_{ij} = 0, \,  r-k+2 \leq i \leq r \,\, \mathrm{and} \,\,  1 \leq j \leq n-r 
\] 
form  a basis for all linear relations on the coefficients of (Tutte) polynomials in $\mathcal{T}_k(n,r).$    
\end{prop}
  
\begin{proof} 
By Corollary \ref{girth}, $T(M)$ is a linear combination of 
Tutte polynomials  $T(1^a0^b1^c 0^d),$ where $a \geq k-1.$   Thus, it suffices to prove Proposition \ref{TPgirth} for the freedom matroids $F(1^a0^b1^c 0^d), \, a \geq k-1.$  But if $a \geq k-1,$ then $c \leq r-k+1$ and $T^{\circ}(U_{c,c+e};x)$ has degree at most $r-k.$ Hence, by Theorem \ref{MI}, 
the coefficients of $T(1^a0^b1^c 0^d)$ satisfy the condition given in the proposition.  
\end{proof}

\section{Classes of matroids}   

In this paper, we focused on the vector spaces $\mathcal{G}(n,r)$ and $\mathcal{T}(n,r)$ spanned by the $\mathcal{G}$-invariants or Tutte polynomials of {\it all} $(n,r)$-matroids.   One could study the vector spaces $\mathcal{G}(\mathcal{C};n,r)$ (respectively, $\mathcal{T}(\mathcal{C};n,r)$) spanned by 
the $\mathcal{G}$-invariants (respectively, Tutte polynomials) of $(n,r)$-matroids in a class $\mathcal{C}$ of matroids.  When $\mathcal{C}$ contains all freedom matroids, then the results in this paper hold for $\mathcal{G}(\mathcal{C};n,r)$ and $\mathcal{T}(\mathcal{C};n,r);$
in particular, they hold when $\mathcal{C}$ is the class of transversal matroids, gammoids, or matroids representable over an infinite field.   For classes of matroids not containing all freedom matroids, then almost nothing is known.  For example, when $\mathcal{C}$ is the class of graphic matroids or the class of binary matroids, 
determining $\dim \mathcal{G}(\mathcal{C};n,r)$ and $\dim \mathcal{T}(\mathcal{C};n,r)$ are interesting open problems.


\section{Appendix.  Tutte polynomials of rank-$3$ freedom matroids on $5$ elements.} 

\showhide{
\begin{figure}
  \centering
  \begin{tikzpicture}[scale=1]
%
\node[inner sep = 0.3mm] (a) at (6.5,-.7) {\small $00111$};
\node[inner sep = 0.3mm] (b) at (7.75,0) {\small $01011$};
  \node[inner sep = 0.3mm] (c) at (9.0,.7) {\small $01101$};
  \node[inner sep = 0.3mm] (d) at (10.25,1.4) {\small $01110$}; 
  \node[inner sep = 0.3mm] (b1) at (6.5,0.7) {\small $10011$};
  \node[inner sep = 0.3mm] (c1) at (7.75,1.4) {\small $10101$}; 
  \node[inner sep = 0.3mm] (d1) at (9,2.1) {\small $10110$};
  \node[inner sep = 0.3mm] (c2) at (6.5,2.1) {\small $11001$}; 
  \node[inner sep = 0.3mm] (d2) at (7.75,2.8) {\small $11010$}; 
  \node[inner sep = 0.3mm] (e) at (6.5,3.5) {\small $11100$};
  \foreach \from/\to in {a/b,b/c,c/d,b/b1,c/c1,d/d1,b1/c1,c1/d1,c1/c2,d1/d2,c2/d2,d2/e} \draw(\from)--(\to);
  \end{tikzpicture}
  \caption{$(\mathcal{S}(5,3),\trianglerighteq)$}
  \label{fig:S53}
\end{figure}}

There are ten rank-$3$ size-$5$ freedom matroids.  Their Tutte polynomials are given by 

\begin{eqnarray*} 
T(F(11100))  & =  &  x^3 + 2x^2 + 3x + 3y + y^2 
\\
T(F(11010))  &  =  &  x^3 + 2x^2 + 2x + xy + 2y + y^2 
\\
T(F(11001))  &  =  &  x^3 + 2x^2 + 2xy + xy^2 
\\
T(F(10110))  &  =  &  x^3 + x^2 + x + xy + x^2y  + y + y^2 
\\
T(F(10101))  &  =  &  x^3 + x^2 + xy + x^2y + xy^2 
\\
T(F(10011))  &  =  &  x^3 + x^2y + x^2 y^2 
\\
T(F(01110))  &  =  &  y x^3 + y x^2 + y x + y^2
\\
T(F(01101))  &  =  &  yx^3 + yx^2 + y^2x  
\\
T(F(01011))  &  =  &  yx^3 + y^2x^2  
\\
T(F(00111))  &  =  &  y^2 x^3. 
\end{eqnarray*}
These polynomials span the vector space $\mathcal{T}(5,3)$ of dimension $7.$  Thus, the space of syzygies or linear relations has dimension $3.$  The following three linear relations give a basis for the space of syzygies: 
\begin{eqnarray*} 
&&
T(F(10101)) - T(F(11001)) - T(F(10110)) + 2 T(F(11010))- T(F(11100)) = 0,
\\
&& T(F(01101))-  T(F(01110)) -T(F(10101)) + T(F(10110))  + T(F(11010)) -  T(F(11100)) = 0, 
\\
&&
T(F(01011)) - T(F(01101)) - T(F(10011)) + 2 T(F(10101))- T(F(11001)) = 0.
\end{eqnarray*} 
Using these relations, we can express the three Tutte polynomials $T(F(\underline{r})),$ where $\underline{r}$ is meet-reducible, in the meet-irreducible basis.  Explicitly, we have 
\begin{eqnarray*} 
&&
T(F(10101)) = T(F(11001)) + T(F(10110)) - 2 T(F(11010))+ T(F(11100)) ,
\\ 
&&
T(F(01101)) = T(F(11001)) + T(F(01110)) - 3 T(F(11010)) + 2T(F(11100)), 
\\
&& 
T(F(01011)) = T(F(01110)) + T(F(10011)) - 2 T(F(10110)) + T(F(11010)). 
\end{eqnarray*}




\begin{thebibliography}{nn}


\bibitem{Brylawski}  T.H. Brylawski, A decomposition theory for combinatorial geometry, Trans. Amer. Math. Soc. 171  (1972) 235--282.  


\bibitem{BryOx}  T. Brylawski, J.G. Oxley,  The Tutte polynomial and its applications. {\it Matroid applications}, 123--225, Cambridge University Press, Cambridge, 1992.     

\bibitem{MR0190045} H.H. Crapo, Single-element extensions of matroids. {\it J. Res. Nat. Bur. Standards Sect. B} 69B (1965), 55--65.

\bibitem{Derksen} H.~Derksen, Symmetric and quasi-symmetric
  functions associated to polymatroids.  {\it J. Algebraic Combin.} 30
  (2009), 43--86.

\bibitem{DerksenFink} H.~Derksen, A.~Fink, Valuative invariants for
  polymatroids. {\it Adv. Math.} 225 (2010), 1840--1892.

\bibitem{Handbook}  J. Ellis-Monaghan, Iain Moffait, Handbook of Tutte polynomials, CRC Press, Boca Raton FL, to appear.   



\bibitem{FalkKung} M.J.~Falk, J.P.S.~Kung, Algebras and valuations
  related to the Tutte polynomial.  In \emph{Handbook of Tutte
    polynomials,} to appear.


\bibitem{ox} J.G.~Oxley, Matroid Theory. Second edition, Oxford
  University Press, Oxford, 2011.

\bibitem{Rota} G.-C. Rota, On the foundations of combinatorial theory. I. Theory of M\"obius functions.  {\it Z. Wahrscheinlichkeittheorie und Verw. Gebiete} 2 (1964), 340--368.

\bibitem{EC1} R.P. Stanley, Enumerative Combinatorics. Volume 1, 2nd edition,  Cambridge 
  University Press, Cambridge, 2001.

\bibitem{EC2} R.P. Stanley, Enumerative Combinatorics, Volume 2.  2nd edition, Cambridge 
  University Press, Cambridge, 2011.








\end{thebibliography}
\end{document}